\documentclass[11pt]{article}

\usepackage{amssymb, amsmath, amsthm, graphicx}
\usepackage[left=1in,top=1in,right=1in]{geometry}

\date{}

\theoremstyle{plain}
      \newtheorem{theorem}{Theorem}[section]
      \newtheorem{lemma}[theorem]{Lemma}
            \newtheorem{claim}[theorem]{Claim}
      \newtheorem{problem}[theorem]{Problem}

\theoremstyle{definition}

\theoremstyle{remark}

\title{Cliques with many colors in triple systems}

\author{Dhruv Mubayi\thanks{Department of Mathematics, Statistics, and Computer Science, University of Illinois, Chicago, IL, 60607 USA.  Research partially supported by NSF grant DMS-1763317. Email: {\tt mubayi@uic.edu.}}
\and Andrew Suk\thanks{Department of Mathematics, University of California at San Diego, La Jolla, CA, 92093 USA. Research partially supported by an NSF CAREER award and an Alfred Sloan Fellowship. Email: {\tt asuk@ucsd.edu}.} }

\begin{document}

\maketitle

\begin{abstract}
Erd\H os and Hajnal constructed a 4-coloring of the triples of an $N$-element set such that every $n$-element subset contains 2 triples with distinct colors, and $N$ is double exponential in $n$. Conlon, Fox and R\"odl asked whether there is some integer $q\ge 3$ and a  $q$-coloring of the triples of an $N$-element set such that every $n$-element subset has 3 triples with distinct colors, and $N$ is double exponential in $n$.  We make the first nontrivial progress on this problem by providing  a $q$-coloring with this property for all $q\geq 9$,  where $N$ is exponential in $n^{2+cq}$ and $c>0$ is an absolute constant.
\end{abstract}

\section{Introduction}

The Ramsey number $r_k(n;q)$ is the minimum integer $N$ such that for any $q$-coloring of the
$k$-tuples of an $N$-element set $V$, there is a subset $A\subset V$ such that all of the $k$-tuples of $A$ have the same color. Estimating $r_3(n;2)$ is one of the most central problems in combinatorics.  The best known bounds, due to Erd\H os, Hajnal
and Rado \cite{ER,EHR}, state that there are positive constants $c$ and $c'$ such that

\begin{equation}\label{3ramsey}
    2^{cn^2} < r_3(n;2) < 2^{2^{c'n}}.
\end{equation}

\noindent Erd\H os conjectured that the upper bound is closer to the truth, namely, $r_3(n;2)$ grows double exponentially in $\Theta(n)$, and he even offered a \$500 reward for a proof.  His conjecture is supported by the fact that a double exponential growth rate is known when we have 4 colors \cite{EH72,EHR}, that is, for fixed $q\geq 4$

\begin{equation}\label{tight}r_3(n;q)= 2^{2^{\Theta(n)}}.\end{equation}

In this paper, we study the following generalization of $r_3(n;q)$.  For integers $n > q \geq t\geq 2$, let $f(n;q,t)$ denote the maximum integer $N$ such that there is a $q$-coloring of the triples of an $N$-element set $V$ with the property that every subset of $V$ of size $n$ induces at least $t$ distinct colors. Thus when $ t= 2$, we have

$$f(n;q,2) = r_3(n;q) - 1,$$

\noindent and for $q\geq t\geq 3$, we have $f(n;q,t) < r_3(n;q)$.  When $t = 3$, Conlon, Fox, and R\"odl raised the following problem \cite{CFR}.

\begin{problem}[Conlon-Fox-R\"odl]\label{cfr}
Is there an integer $q\geq 3 $ and a positive constant $c$ such that $f(n;q,3)  > 2^{2^{cn}}$ holds for all $n > 2$?

\end{problem}
 
A simple application of the Probabilistic Method (see \cite{AS}) shows that $f(n;q,3)> 2^{cn^2}$, where $c = c(q)$.  Our main result is the following.

\begin{theorem}\label{main2}

There is an absolute constant $c>0$ such that for all integers $n > q \geq 9$, 

$$f(n;q,3) \geq 2^{n^{2 + c\cdot q  }}.$$

\end{theorem}

\noindent For larger values of $t$, we show the following.

\begin{theorem}\label{main}

Given integers $q \geq t \geq 2$, there is an $n_0 =n_0(q,t)$ such that for all integers $n > n_0$,

$$f(n;q,t) \geq 2^{n^{\log(q/(t-1))}/4}.$$

\end{theorem}

Both proofs are based on a stepping-up argument introduced by Erd\H os and Hajnal \cite{EH72}. We start with the proof of Theorem \ref{main} in the next section, as it is a direct application of the stepping-up method.  The proof of Theorem \ref{main2} combines a more general stepping-up argument with induction, and is given in Section~\ref{pfmain2}. Throughout this paper, all logarithms are in base 2.

\section{Forcing many colors}

In this section, we prove Theorem \ref{main}.  We will need the following lemma.

\begin{lemma}\label{graph}

Given integers $q\geq t \geq 2$, there is an integer $m_0$ such that the following holds.  For every $m \geq m_0$, there is a $q$-coloring $\phi$ of the pairs of $U = \{0,1,\ldots, \lfloor (q/(t-1))^{m/4}\rfloor - 1\}$ such that every subset of size $m$ induces at least $t$ distinct colors.

\end{lemma}

\begin{proof}  Given $q \geq t \geq 2$, let $m_0 = m_0(q,t)$ be a sufficiently large integer that will be determined later.
Color the pairs of  $U = \{0,1,\ldots, \lfloor (q/(t-1))^{m/4}\rfloor$ uniformly independently at random with colors $\{\alpha_1,\ldots, \alpha_q\}$.  Let $X$ denote the number of subsets $A\subset U$ of size $m$ that have less than $t$ distinct colors among their pairs.  Then we have 

$$\mathbb{E}[X] \leq \binom{|U|}{m} \binom{q}{t-1}\left(\frac{t-1}{q}\right)^{\binom{m}{2}} \leq \left(\frac{q}{t-1}\right)^{m^2/4} q^{t-1} \left(\frac{t-1}{q}\right)^{m^2/2} = q^{t-1}\left(\frac{q}{t-1}\right)^{-m^2/4}.$$

\noindent By setting $m_0 = m_0(q,t)$ sufficiently large, we have for all $m \geq m_0$, $\mathbb{E}[X] < 1.$  Hence, there is a $q$-coloring $\phi:\binom{U}{2}\rightarrow \{\alpha_1,\ldots, \alpha_q\}$ such that every subset $A\subset U$ of size $m$ has at least $t$ distinct colors among its pairs.\end{proof}

\begin{proof}[Proof of Theorem \ref{main}] Given $q \geq t\geq 2$, let $n_0 = n_0(q,t)$ be a sufficiently large integer that will be determined later.  Set $M =  \lfloor (q/(t-1))^{m/4}\rfloor$, $U  = \{0,1,\ldots, M-1\}$, and let $\phi:\binom{U}{2} \rightarrow \{\alpha_1,\ldots, \alpha_q\}$ be a $q$-coloring of the pairs of $U$ with the properties described in Lemma \ref{graph}. Set $V = \{0,1,\ldots, 2^M - 1\}$.  In what follows, we will use $\phi$ to define a $q$-coloring $\chi:\binom{V}{3}\rightarrow \{\alpha_1,\ldots, \alpha_q\}$ of the triples of $V$ with the desired properties.

For each $v \in V$, write $v=\sum_{i=0}^{M-1}v(i)2^i$ with $v(i) \in \{0,1\}$ for each $i$. For $u \not = v$, let $\delta(u,v) \in U$ denote the largest $i$ for which $u(i) \not = v(i)$.  Notice that we have the following stepping-up properties (see \cite{GRS})

\begin{description}

\item[Property I:] For every triple $u < v < w$, $\delta(u,v) \not = \delta(v,w)$ .

\item[Property II:] For $v_1 < \cdots < v_r$, $\delta(v_1,v_{r}) = \max_{1 \leq j \leq r-1}\delta(v_j,v_{j + 1})$.

\end{description}

Using $\phi:\binom{U}{2}\rightarrow \{\alpha_1,\ldots, \alpha_q\}$, we define $\chi:\binom{V}{3}\rightarrow \{\alpha_1,\ldots, \alpha_q\}$ as follows.  For vertices $v_1 < v_2 < v_3$ in $V$ and $\delta_i = \delta(v_i,v_{i + 1})$, we define $\chi(v_1,v_2,v_3) = \alpha_j$ if and only if $\phi(\delta_1,\delta_2) = \alpha_j$. We now need the following lemma.

\begin{lemma}\label{stepdown}
For $m\geq 2$ set $n = 2^m$.  Then for any set of $n$ vertices $v_1,\ldots, v_n \in V$, where $v_1 < \cdots < v_n$, there is a subset $B \subset \{\delta(v_i,v_{i + 1}): 1 \leq i \leq n-1\}$ with at least $m$ distinct elements such that for each pair $(\delta_r,\delta_s) \in \binom{B}{2}$, there is a triple $v_i < v_j < v_k$ in $\{v_1,\ldots, v_n\}$ such that $\chi(v_i,v_j,v_k) = \phi(\delta_r,\delta_s)$.
\end{lemma}

\begin{proof}  We proceed by induction on $m$.  The base case $m = 2$ follows from Property I.  For the inductive step, assume that the statement holds for all $m' < m$.  Let $v_1,\ldots, v_n \in V$ such that $v_1 < \cdots  < v_n$ and $n = 2^m$.  Let $\delta_i = \delta(v_i,v_{i +1})$, for $i = 1,\ldots, n-1$.  Set $\delta_{w} = \max\{\delta_i: 1 \leq i \leq n-1\}$ and notice that, by Properties I and II above, $\delta_w > \delta_i$ for all $i\neq w$.  Set $S = \{v_1,\ldots, v_w\}$ and $T = \{v_{w + 1},\ldots, v_n\}$.  Then either $|S|$ or $|T|$ has size at least $2^{m-1}$.  Without loss of generality, we can assume that $|S|\geq 2^{m-1}$ since a symmetric argument would follow otherwise.  By the induction hypothesis, there is a subset $B_0  \subset \{\delta_1,\ldots, \delta_{w-1}\}\subset U$ with at least $m-1$ distinct elements and for each pair $(\delta_r,\delta_s) \in \binom{B_0}{2}$, there is a triple $v_i < v_j < v_k$ in $S$ such that

$$\chi(v_i,v_j,v_k) = \phi(\delta_r,\delta_s).$$

Set $B = \{\delta_w\}\cup B_0$, which implies $|B| \geq m$.  Then notice that for each pair $(\delta_w, \delta_r)$, where $\delta_r \in B_0$, by Property I above, we have

$$\chi(v_r,v_{r + 1},v_{w + 1}) = \phi(\delta_w,\delta_r).$$

\noindent Hence $B\subset U$ has the desired properties, and this completes the proof of the claim.\end{proof}

Set $n_0  = \lceil 2^{m_0}\rceil$ where $m_0$ is defined in Lemma \ref{graph}.  Then for all $n > n_0$ we have $m > m_0$.  Thus, by Lemma \ref{graph} and Lemma \ref{stepdown}, any set of $n$ vertices in $V$ induces at least $t$ distinct colors with respect to $\chi$.  Since $|V| = 2^{(q/(t-1))^{m/4}}$ and $n = 2^m$, we have $|V| = 2^{n^{\log(q/(t-1))}/4}$.\end{proof}

\section{Forcing three colors}\label{pfmain2}

In this section, we prove Theorem \ref{main2}.  We will need the following lemma.

\begin{lemma}\label{3random}

Let $r > 3$ and set $V_3 = \{0,1,\ldots, \lfloor 2^{r^2/24}\rfloor - 1\}$.  Then there is a 3-coloring $\phi_3:\binom{V_3}{3}\rightarrow \{\beta_1,\beta_2, \beta_3\}$ of the triples of  $V_3$ such that every subset of size $r$ induces at least three distinct colors.

\end{lemma}

\noindent We omit the proof of Lemma \ref{3random} as it follows by the same probabilistic argument used for Lemma~\ref{graph}.  Hence, Lemma \ref{3random} implies that $f(n;3,3) \geq 2^{n^2/24}.$  Together with the following recursive formula, Theorem \ref{main2} quickly follows.

\begin{theorem}\label{base}

For integers $n > q  \geq 9$, we have

$$f(n;q,3) \geq \left(f(\lfloor n/\log n\rfloor, q-6,3)\right)^{n^{1/4}/2}.$$
\end{theorem}

\noindent We will also need the following lemma, whose proof is also omitted since it follows from the same probabilistic argument as in Lemma \ref{graph}.

\begin{lemma}\label{2random}

Let $s > 3$ and set $V_2 = \{0,1,\ldots , \lfloor 2^{s/4}\rfloor\}$.  Then there is a 3-coloring $\phi_2:\binom{V_2}{2}\rightarrow \{\alpha_1,\alpha_2,\alpha_3\}$ of the pairs of $V_2$ such that every subset of size $s$ induces at least three distinct colors.

\end{lemma}

\begin{proof}[Proof of Theorem \ref{base}]  Given $n > q \geq 9$, let $r = \lfloor n/\log n\rfloor$ and $s = \lfloor \log n\rfloor$.  Set $N_2 = \lfloor 2^{s/4}\rfloor$, $N_3 = f(r;q-6,3)$, and 

$$V_2 = \{0,1,\ldots, N_2 - 1\}\hspace{.5cm}\textnormal{ and }\hspace{.5cm}V_3 = \{0,1,\ldots, N_3 - 1\}.$$  

Using Lemma \ref{2random}, we obtain $\phi_2:\binom{V_2}{2}\rightarrow \{\alpha_1,\alpha_2,\alpha_3\}$ such that every subset of $V_2$ of size $s$ induces at least three colors.  Likewise, by definition of $f(r, q-6,3)$, we obtain  $\phi_3:\binom{V_3}{3}\rightarrow \{\beta_1,\ldots,\beta_{q-6}\}$ such that every subset of $V_3$ of size $r$ induces at least three distinct colors.   We now apply the following more general stepping-up procedure.

Set $N =N_3^{N_2}$ and $V = \{0,1,\ldots, N-1\}$.  For each $v \in V$, write $v=\sum_{i=0}^{N_2-1}v(i)(N_3)^i$ with $v(i) \in V_3$ for each $i$. For $u,v \in V$ with $u < v$, let $\delta(u,v) \in V_2$ denote the largest $i$ for which $u(i) \not = v(i)$.  Notice that we no longer have Property I from the previous stepping-up procedure, but we do have the following properties.

\begin{description} 

\item[Property II:] For $v_1 < \cdots < v_r$, $\delta(v_1,v_{r}) = \max_{1 \leq j \leq r-1}\delta(v_j,v_{j + 1})$.

\item[Property III:] For $v_1 < v_2 < v_3$ such that $\delta(v_1,v_2) = \delta(v_2,v_3) = i$, $v_1(i) < v_2(i) < v_3(i)$.

\end{description}

Using $\phi_2$ and $\phi_3$, we define $\chi:\binom{V}{3}\rightarrow \{\gamma_1,\ldots, \gamma_q\}$ as follows.  For vertices $v_1 < v_2 < v_3$ in $V$, let $\delta_1 = \delta(v_1,v_2)$ and $\delta_2 = \delta(v_2,v_3)$.   Then for $i \in \{1,2,3\}$,

\begin{itemize}

    \item set $\chi(v_1,v_2,v_3) = \gamma_i$ if and only if $\delta_1 > \delta_2$ and $\phi_2(\delta_1,\delta_2) = \alpha_i,$
    
      \item set $\chi(v_1,v_2,v_3) = \gamma_{3 + i}$ if and only if $\delta_1 < \delta_2$ and $\phi_2(\delta_1,\delta_2) = \alpha_i,$
    
 \end{itemize}
 
 \noindent and for $i \in \{1,\ldots, q-6\}$,
 
 \begin{itemize}
       
\item set $\chi(v_1,v_2,v_3) = \gamma_{6 + i}$ if and only if $\delta_1 = \delta_2  = j$ and $\phi_3(v_1(j), v_2(j), v_3(j)) = \beta_i,$

\end{itemize}

Notice that $n \geq \max\{s\cdot r, 2^s\}$.  We claim that any set of $n$ vertices $v_1,\ldots, v_n \in V$ induces at least 3 distinct colors with respect to $\chi$.  For sake of contradiction, let $A = \{v_1,\ldots, v_n\}\subset V$ such that $v_1 < \cdots < v_n$ and $\chi(v_i,v_j,v_k) \in\{\gamma_x,\gamma_y\}$ for all triples $(v_i, v_j,v_k) \in \binom{A}{3}$.  Set $\delta_i = \delta(v_i,v_{i + 1})$ for $i = 1,\ldots, n-1$. The proof now falls into the following cases.

\medskip

\noindent \emph{Case 1.}  Suppose $\gamma_x,\gamma_y \in \{\gamma_1,\gamma_2,\gamma_3\}$.  Then we have $\delta_1 > \delta_2 > \cdots > \delta_{n-1}$.  However, $\delta_i \in U  = \{0,1,\ldots, \lfloor 2^{s/4}\rfloor - 1\}$ and $n = 2^{s}$ which is a contradiction. A similar argument follows if $\gamma_x,\gamma_y  \in \{\gamma_4,\gamma_5,\gamma_6\}$.

\medskip

\noindent \emph{Case 2.}  Suppose $\gamma_x,\gamma_y \in \{\gamma_7,\ldots, \gamma_{q-6}\}$.  Then we must have $\delta_1 = \cdots = \delta_{n-1} = i$ and $v_1(i)  < \cdots < v_{n-1}(i)$.  Since $n \geq r$, by definition of $\chi$ and $\phi_3$, the set $\{v_1,\ldots, v_n\}$ induces at least three distinct colors, contradiction.

\medskip

\noindent \emph{Case 3.} Suppose $\gamma_x \in \{\gamma_1,\gamma_2,\gamma_3\}$ and $\gamma_y \in \{\gamma_4,\gamma_5,\gamma_6\}.$  Then in this case, for any triple $v_i < v_j < v_k$, we have $\delta(v_i,v_j) \neq \delta(v_j,v_k)$ and $\phi_2(\delta(v_i,v_j) ,\delta(v_j,v_k)) = \alpha_z$ for some fixed $z$.   Set $\delta_{w} = \max\{\delta_i: 1 \leq i \leq n-1\}$ and notice that, by Property II above, $\delta_w > \delta_i$ for all $i\neq w$. Therefore, a straight-forward adaptation of Lemma \ref{stepdown} gives us the following claim.

\begin{claim}
For $s\geq 2$, any set of $2^s$ vertices $v_1,\ldots, v_{2^s} \in V$, with the properties described above, there is a subset $B \subset \{\delta(v_i,v_{i + 1}): 1 \leq i \leq 2^s-1\}$ with at least $s$ distinct elements such that $\phi_2(\delta_i,\delta_j) = \alpha_z$ for every pair $(\delta_i,\delta_j) \in \binom{B}{2}$.

\end{claim}

\noindent However, this contradicts Lemma \ref{2random}.
\medskip

\noindent \emph{Case 4.}  Suppose $\gamma_x \in \{\gamma_1,\ldots, \gamma_6\}$ and $\gamma_y \in \{\gamma_7,\ldots, \gamma_q\}$.  Without loss of generality, we can assume that $\gamma_x = \gamma_1$ and $\gamma_y = \gamma_7$ since a symmetric argument would follow otherwise.  Notice that there is an integer $w_1 \in \{1,\ldots, r\}$ such that $\delta(v_1,v_{w_1}) > \delta(v_{w_1},v_{w_1 + 1})$.  Indeed, otherwise if $\delta_1 = \cdots =\delta_r$, by the definition of $\chi$ and the properties of $\phi_3$ described above, the set $\{v_1,\ldots, v_r\}$ induces at least three distinct colors with respect to $\chi$, contradiction.

The same argument shows that there must be an integer $w_2 \in \{w_1+1\ldots, w_1 + r\}$ such that $\delta(v_{w_1},v_{w_2}) > \delta(v_{w_2},v_{w_2  + 1})$.  Since $n \geq s\cdot r$, a repeated application of the argument above shows that there are integers $w_1 < \cdots < w_{s-1}$, such that

$$\delta(v_1,v_{w_1}) > \delta(v_{w_1},v_{w_2}) > \delta(v_{w_2},v_{w_3}) > \cdots >  \delta(v_{w_{s-1}},v_{w_{s - 1} + 1}).$$
 
\noindent By Property II, $\chi$ colors every triple in $\{v_1,v_{w_1},\ldots, v_{w_{s-1}}, v_{w_{s-1} + 1}\}$ with color $\gamma_1$.  However, this implies that the set 

$$S = \{\delta(v_1,v_{w_1}),\delta(v_{w_1},v_{w_2}) , \ldots , \delta(v_{w_{s-2}},v_{w_{s-1}}), \delta(v_{w_{s-1}},v_{w_{s - 1} + 1})\}\subset U,$$

\noindent has the property that $|S| = s$ and $\phi_2:\binom{S}{2} \rightarrow \alpha_1$, which is a contradiction.  Since $|V| = N_3^{N_2},$

$$f(n;q,3) \geq |V| \ge  \left(f(\lfloor n/\log n\rfloor;q-6,3)\right)^{n^{1/4}/2}.$$

\noindent This completes the proof of Theorem \ref{base}.\end{proof}

Combining Theorem \ref{base} with the fact that $f(n;3,3) > 2^{n^2/24}$ gives the following.

\begin{theorem}
For fixed $q \geq 3$ and for all $n  > 3$ we have

$$f(n;q,3) > 2^{n^{2 + \frac{1}{4}\left\lfloor\frac{q-3}{6}\right\rfloor - o(1)}}.$$

\end{theorem}

\end{document}